\documentclass[12pt]{article}

\usepackage{diagbox}
\usepackage{mathtools}
\usepackage{bbm}
\usepackage{latexsym}
\usepackage{epsfig}
\usepackage{amsmath,amsthm,amssymb,enumerate}

\usepackage[a-1b]{pdfx}
\usepackage{hyperref}

\usepackage{color}

\def\reals{\mathbb R}

\newcounter{rot}

\allowdisplaybreaks

\def\nn{\nonumber}
\def\a{\alpha} \def\b{\beta} \def\d{\delta} \def\D{\Delta}
\def\e{\varepsilon}    \def\g{\gamma}
\def\G{\Gamma}  
\def\z{\zeta} \def\th{\theta}    \def\l{\lambda}
   \def\p{\pi}
   
\def\t{\tau} \def\om{\omega}

\def\cP{{\cal P}}

\newtheorem{theorem}{Theorem}
\newtheorem{lemma}[theorem]{Lemma}

\def\cX{{\mathcal X}}

\newcommand{\brac}[1]{\left(#1\right)}

\newcommand{\bfrac}[2]{\left(\frac{#1}{#2}\right)}

\def\cE{{\mathcal E}}

\newcommand{\set}[1]{\left\{#1\right\}}

\def\E{\mathbb{E}}
\def\Var{{\mathbb Var}}
\def\Pr{\mathbb{P}}

\def\cF{{\cal F}}
\newcommand{\ignore}[1]{}

\def\cX{{\mathcal X}}

\newcommand{\beq}[2]{\begin{equation}\label{#1}#2\end{equation}}
\newcommand{\mults}[1]{\begin{multline*}#1\end{multline*}}

\def\cG{\mathcal{G}}

\usepackage{tikz}
\usetikzlibrary{decorations}
\usetikzlibrary{decorations.pathreplacing}
\usetikzlibrary{shapes.misc}
\usetikzlibrary{arrows}

\begin{document}
\author{Alan Frieze\thanks{Research supported in part by NSF grant DMS1952285 } and Wesley Pegden\thanks{Research supported in part by NSF grant DMS1363136 }\\Department of Mathematical Sciences\\Carnegie Mellon University\\Pittsburgh PA 15213}

\date{}
\title{Spanners in randomly weighted graphs: independent edge lengths}
\maketitle
\begin{abstract}
Given a connected graph $G=(V,E)$ and a length function $\ell:E\to \reals$ we let $d_{v,w}$ denote the shortest distance between vertex $v$ and vertex $w$. A $t$-spanner is a subset $E'\subseteq E$ such that if $d'_{v,w}$ denotes shortest distances in the subgraph $G'=(V,E')$ then $d'_{v,w}\leq t d_{v,w}$ for all $v,w\in V$. We show that for a large class of graphs with suitable degree and expansion properties with independent exponential mean one edge lengths, there is w.h.p.~a 1-spanner that uses $\approx \frac12n\log n$ edges and that this is best possible.  In particular, our result applies to the random graphs $G_{n,p}$ for $np\gg \log n$.
\end{abstract}
\section{Introduction}
Given a connected graph $G=(V,E)$ and a length function $\ell:E\to \reals$ we let $d_{v,w}$ denote the shortest distance between vertex $v$ and vertex $w$. A $t$-spanner is a subset $E'\subseteq E$ such that if $d'_{v,w}$ denotes shortest distances in the subgraph $G'=(V,E')$ then $d'_{v,w}\leq t d_{v,w}$ for all $v,w\in V$. In general, the closer $t$ is to one, the larger we need $E'$ to be relative to $E$. Spanners have theoretical and practical applications in various network design problems. For a recent survey on this topic see Ahmed et al \cite{Aetal}. Work in this area has in the main been restricted to the analysis of the worst-case properties of spanners. In this note, we assume that edge lengths are random variables and do a probabilistic analysis.

Suppose that $G=([n],E)$ is almost regular in that 
\beq{degs}{
(1-\th)dn\leq \d(G)\leq \D(G)\leq (1+\th)dn
}
where $1\geq d\gg \tfrac{\log\log n}{\log^{1/2}n}$ and $\th=\tfrac{1}{\log^{1/2}n}$.  Here $\d,\D$ refer to minimum and maximum degree respectively. 

We will also assume either that $d>1/2$ or 
\beq{0}{
|E(S,T)|\geq \psi |S|\,|T|\text{ for all }|S|,|T|\geq \th n.
}
Here $\psi=\frac{\om\log\log n}{\log^{1/2}n}\leq d$ where $\om=\om(n)\to\infty$ as $n\to\infty$ and $E(S,T)$ denotes the set of edges of $G$ with one end in $S\subseteq [n]$ and the other end in $T\subseteq [n],\,S\cap T=\emptyset$.

Let $\cG(d)$ denote the set of graphs satisfying the stated conditions, \eqref{degs} and \eqref{0}. We observe that $K_n\in\cG(1)$ and that w.h.p. $G_{n,p}\in\cG(p)$, as long as $np\gg \log n$. The weighted perturbed model of Frieze \cite{FW} where randomly weighted edges are added to a randomly weighted $dn$-regular graph also lies in $\cG(d)$.
 
Suppose that the edges $\set{i,j}$ of $G$ are given independent lengths $\ell_{i,j},1\leq i<j\leq n$ that are distributed as the exponential mean one random variable, denoted by $E(1)$. In general we let $E(\l)$ denote the exponential random variable with mean $1/\l$.

When $G=K_n$, Janson \cite{Jan} proved the following: W.h.p. and in expectation
\beq{j1}{
d_{1,2}\approx \frac{\log n}{n};\quad \max_{j>1}d_{1,j}\approx \frac{2\log n}{n};\quad \max_{i,j}d_{i,j}\approx \frac{3\log n}{n}.
}
Here (i) $A_n\approx B_n$ if $A_n=(1+o(1))B_n$ and (ii) $A_n\gg B_n$ if $A_n/B_n\to\infty$, as $n\to\infty$.

It follows that w.h.p. the length of the longest edge in any shortest path is at most $L=\frac{(3+o(1))\log n}n$. It follows further that w.h.p. if we let $E'$ denote the set of edges of length at most $L$ then this is a 1-spanner of size $O(n\log n)$. We tighten this and extend it to graphs in the class $\cG(d)$.
\begin{theorem}\label{th1}
Let $G\in\cG(d)$ or let $G$ be a $dn$-regular graph with $d>1/2$ where the lengths of edges are independent exponential mean one. The following holds w.h.p.
\begin{enumerate}[(a)]
\item The minimum size of a 1-spanner is asymptotically equal to $\frac{1}{2}n\log n$.
\item If $2\leq \l=O(1)$ then a $\l$-spanner requires at least $\frac{n\log n}{601d\l}$ edges.
\end{enumerate}
\end{theorem}
A companion paper deals with $(1+\e)$-spanners in embeddings of $G_{n,p}$ in $[0,1]^2$ as studied by Frieze and Pegden \cite{FP}. Here we choose $n$ random points $\cX=\set{X_1,X_2,\ldots,X_n}$ in $[0,1]^2$ and connect a pair $X_i,X_j$ with probability $p$ by an edge of length $|X_i-X_j|$.
\section{Proof of Theorem \ref{th1}}
The proof of Theorem \ref{th1} uses a few parameters. We will list some of them here for easy reference:
\begin{align*}
&\th=\frac{1}{\log^{1/2}n}; \qquad k_0=\log n;\qquad k_1=\th n;\qquad \a=1-2\th.\\
&\ell_0=\frac{(1+\sqrt{\th})\log n}{dn};\qquad \ell_1=\frac{5\log n}{dn};\qquad \ell_2=\ell_0-\frac{(\log\log n)^2}{dn};\qquad \ell_3=\frac{\log n}{200\l dn}.
\end{align*} 
We also use the Chernoff bounds for the binomial $B(n,p)$: for $0\leq \e\leq 1$,
\begin{align*}
\Pr(B(n,p)\leq (1-\e)np)&\leq e^{-\e^2np/2}.\\
\Pr(B(n,p)\geq (1+\e)np)&\leq e^{-\e^2np/3}.\\
\Pr(B(n,p)\geq \a np)&\leq\bfrac{e}{\a}^{\a np}.
\end{align*}
It will only be in Section \ref{partb} that we will need to use condition \eqref{0}.
\subsection{Lower bound for part (a)}
We identify sets $X_v$ (defined below) of size $\approx\log n$ such that w.h.p. a 1-spanner must contain $X_v$ for $n-o(n)$ vertices $v$. The sets $X_v$ are the edges from $v$ to its nearest neighbors. If an edge $\set{v,x}$ is missing from a set $S\subseteq E(K_n)$ then a path from $v$ to $x$ must go to a neighbor $y$ of $v$ and then traverse $K_n-v$ to reach $x$. Such a path is likely to have length at least the distance promised by \eqref{j1}, scaled by $d^{-1}$.

We first prove the following: 
\begin{lemma}\label{lem1}
Fix $v,w_1,w_2,\ldots,w_\ell$ for $\ell=O(\log n)$ and let $\a=1-2\th$. Then,
\[
\Pr\brac{\exists 1\leq i\leq \ell:d_{v,w_i}\leq \frac{\a\log n}{dn}}=o(1).
\]
\end{lemma}
\begin{proof}
There are at most $((1+\th)dn)^{k-1}$ paths using $k$ edges that go from vertex $v$ to vertex $w_i,1\leq i\leq \ell$. The random variable $E(1)$ dominates the uniform $[0,1]$ random variable $U_1$. We write this as $E(1)\succ U_1$.  As such we can couple each edge weight with a lower bound given by a copy of $U_1$. The length of one of these $k$-edge paths is then at least the sum of $k$ independent copies of $U_1$. The fraction $x^k/k!$ is an upper bound on the probability that this sum is at most $x$ (tight if $x\leq 1$). Therefore,
\begin{align}
\Pr\brac{\exists 1\leq i\leq \ell:d_{v,w_i}\leq x=\frac{\a\log n}{dn}}&\leq  \ell\sum_{k=1}^{n-1}  ((1+\th)dn)^{k-1}\frac{x^k}{k!} \label{eq1}\\
&\leq \frac{\ell}{dn}\sum_{k=1}^{n-1}\bfrac{e^{1+\th}\a\log n}{k}^k=
\frac{\ell}{dn}\sum_{k=1}^{10\log n}\bfrac{e^{1+\th}\a\log n}{k}^k +O(n^{-10})\nonumber\\
&\leq \frac{10\ell \log n}{dn^{1-\a e^\th}}+o(1)=o(1).\nonumber
\end{align}

\end{proof}
For a vertex $v\in [n]$, let 
\[
A_v=\set{w\neq v:\ell_{v,w}\leq \frac{\log n}{dn}}.
\] 
\begin{lemma}\label{lem2}
W.h.p. $|A_v|\leq 4\log n$ for all $v\in [n]$. 
\end{lemma}
\begin{proof}
We have, from the Chernoff bounds and $E(1)\succ U_1$ that
\beq{eq2}{
\Pr(|A_v|\geq 4\log n)\leq\Pr\brac{Bin\brac{(1+\th)dn,\frac{\log n}{dn}}\geq 4\log n}\leq \bfrac{e(1+\th)}{4}^{4\log n}=o(n^{-1}).
}
The lemma follows from the union bound, after multiplying the RHS of \eqref{eq2} by $n$.
\end{proof}
For $v\in[n]$, let $\d_v$ be the distance from $v$ to its nearest neighbor. Let 
\[
B=\set{v:\d_v\geq \frac{\log^{1/2}n}{dn}}.
\]
\begin{lemma}\label{lem3}
$|B|\leq ne^{-\log^{1/3}n}$ w.h.p.
\end{lemma}
\begin{proof}
We have
\[
\E(|B|)\leq n\brac{\exp\set{-\frac{\log^{1/2}n}{dn}}}^{(1-\th)dn}= ne^{-(1-\th)\log^{1/2}n}.
\]
The lemma follows from the Markov inequality.
\end{proof}
Let 
\[
X_v=\set{e=\set{v,x}:\ell(e)\leq\d_v+\frac{\a\log n}{dn}}.
\]
\begin{lemma}\label{lem4}
Let $S\subseteq E(K_n)$ define a 1-spanner. Then w.h.p. $S\supseteq X_v$ for all but $o(n)$ vertices $v$.
\end{lemma}
\begin{proof}
  Let $G_S=([n],S)$ and suppose that $v\notin B$. Then
  \begin{equation}
    \d_v+\frac{\a\log n}{dn}< \frac{\log^{1/2}n}{dn}+\frac{\a\log n}{dn}<\frac{\log n}{dn}
  \end{equation}
  and so $X_v\subseteq \{v\}\times A_v$ and in particular $|X_v|\leq 4\log n$ w.h.p. by Lemma \ref{lem2}.

  If $G_S$ does not contain an edge $e=\set{v,x}\in X_v$, then
  the $G_S$-distance from $v$ to $x$ is then w.h.p. at least 
\beq{eq3}{
\d_v+\frac{\a\log n}{dn}>d_{v,x}.
}
To obtain \eqref{eq3} we have used Lemma \ref{lem1} applied to $K_n-v$ with $x$ replacing $v$ and $w_1,w_2,\ldots,w_\ell$ being the remaining neighbors of $v$ in $K_n$.

So, if 
\[
C=\set{v\notin B:\exists \text{1-spanner }S\not\supseteq X_v},
\]
then $\E(|C|)=o(n)$.

Any 1-spanner must contain $X_v,\,v\in [n]\setminus (B\cup C)$ and the lemma follows from the Markov inequality.
\end{proof}
Now $|X_v|$ dominates $Bin\brac{(1-\th)dn,1-\exp\set{-\frac{\a\log n}{dn}}}$ and so by the Chernoff bounds
\[
\Pr\brac{|X_v|\leq (1-\e)\a\log n+O\bfrac{\log^2n}{n}}\leq e^{-\e^2\a\log n/(2+o(1))}=o(1)\text{ for }\e=\log^{-1/3}n.
\]
Applying Lemma \ref{lem4} we see that w.h.p. a 1-spanner contains at least $\frac{1-o(1)}{2}n\log n$ edges. The factor 2 comes from the fact that $\set{v,w}$ can be in $X_v\cap X_w$. (In this case the edge $\set{v,w}$ contributes twice to the sum of the $|A|_v|$'s.)  Note that we do not need \eqref{0} to prove the lower bound.
\subsection{Upper bound for part (a)}\label{partb}
Let $\ell_0=\frac{(1+\sqrt{\th})\log n}{dn}$ and $\ell_1=\frac{5\log n}{dn}$ and $E_0=\set{e:\ell(e)\leq \ell_0}$. Now $|E(G)|\in (1\pm\th)dn^2/2$ and so the Chernoff bounds imply that w.h.p. $|E_0|\approx \frac12n\log n$ and our task is to show that adding $o(n\log n)$ edges to $E_0$ gives us  a 1-spanner w.h.p. We will do this by showing that w.h.p. there are only $o(n\log n)$ edges $e$ with $\ell(e)>\ell_0$ that are the shortest path between their endpoints.  Adding these $o(n\log n)$ edges to $E_0$ creates a 1-spanner, since every edge on a shortest path in a graph is itself a shortest path between its endpoints.

Janson \cite{Jan} analysed the performance of Dijkstra's \cite{Dijk} algorithm on the complete graph $K_n$ with exponential edge-weights; we will adapt his argument to our setting on a graph $G$ satisfying conditions \eqref{degs} and \eqref{0}.

In particular, we analyze Dijkstra's algorithm for shortest paths from vertex 1 where edges have exponential weights. 
Recall that after $i$ steps of the algorithm we have a tree $T_i$ and a set of values $d_{v},v\in [n]$ such that for $u\in T_i$, $d_u$ is the length of the shortest path from 1 to $u$. For $v\notin T_i$, $d_v$ is the length of the shortest path from 1 to $v$ that follows a path from 1 to $u\in T_i$ and then uses the edge $\set{u,v}$. Let $\d_i=\max\set{v\in T_i:d_v}$. 

The constraints on the length $l(u,v)$ of the edge $\{u,v\}$ for $u\in T_i,v\notin T_i$ are that $d_u+l(u,v)\geq \d_i$ or equivalently that $l(u,v)\geq \d_i-d_u$.  Fixing $T_i$ and the lengths of edges within $T_i$ or its complement, every set of lengths $\{l(u,v)\}_{\substack{u\in T_i\\ v\notin T_i}}$ satisfying these constraints would give the same history of the algorithm to this point. Due to the memoryless property of the exponential distribution we then have that $l(u,v)=\d_i-d_u+E_{u,v}$ where $E_{u,v}$ is a mean-1 exponential, independent of all other $E(u',v')$.

Thus the Dijkstra algorithm is equivalent in distribution to the following discrete-time process:
\begin{itemize}
\item Set $v_1=1$, $T_1=\set{1}$.
\item Having defined $T_{i}$, associate a mean-1 exponential $E_{u,v}$ to each edge $\set{u,v}\in E(T_{i},\bar T_i)$ that is independent of the process to this point.  Define $e_{i+1}$ to be the edge $\set{u,v}\in E(T_i,\bar T_i)$ minimizing $\d_i+E_{u,v}$, and define $v_{i+1}$ to be the vertex for which $e_{i+1}=\set{v_j,v_{i+1}}$ for some $v_j\in T_i$.  Finally define $d_{v_{i+1}}$ by $\d_i+E_{v_i,v_j}$.
\end{itemize}

Finally, note that, as the minimum of $r$ rate-1 exponentials is an exponential of rate $r$, this is equivalent in distribution to the following process:

\begin{itemize}
\item Set $v_1=1$, $T_1=\set{1}$.
\item Having defined $v_i,T_{i},$ define a vertex $v_{i+1}$ by choosing an edge $e_{i+1}=\{v_j,v_{i+1}\}$ ($j\leq i$) uniformly at random from $E(T_i,\bar T_i)$, set $T_{i+1}=T_i\cup \set{v_{i+1}}$, and define $d_{1,v_{i+1}}=d_{1,v_i}+E_{i}^{\gamma_{i}}$ where $E_{i}^{\gamma_{i}}$ is an (independent) exponential random variable of rate $\gamma_{i}=E(T_i,\bar T_i)$.
\end{itemize}
It follows that 
\[
\E(d_{1,m})=S_m:=\sum_{i=1}^{m-1}\E\bfrac{1}{\g_i}\quad\text{and}\quad \Var(d_{1,m})=\sum_{i=1}^{m-1}\E\bfrac{1}{\g_i^2}.
\]
Observe that we have
\[
(1-\th)i(dn-i)\leq\g_i\leq (1+\th)idn\quad\text{w.h.p.}
\]
and so for $1\leq i\leq \th n$ we have
\[
\g_i=idn(1+\z_i)\text{ where }|\z_i|=O(\th),\quad\text{w.h.p.}
\]
Also, we have
\[
\g_i=(n-i)dn(1+\z_i)\text{ where }|\z_i|=O(\th)\quad\text{w.h.p.}
\]
for $n-\th n\leq i\leq n$.

It follows that
\beq{Sn2}{
S_{\th n}=(1+O(\th))\sum_{i=1}^{\th n}\frac{1}{dni}=\frac{\log n}{dn}+ O\bfrac{\log^{1/2}n}{n}\quad\text{w.h.p.}
}
\begin{lemma}\label{maxdij}
W.h.p. $\max_{i,j}d_{i,j}\leq \ell_1=\frac{5\log n}{dn}$.
\end{lemma}
\begin{proof}
Following \cite{Jan}, let $k_1=\th n$ and $Y_i=E_{i}^{\gamma_{i}},1\leq i<n$ so that $Z_1=d_{1,k_1}=Y_1+Y_2+\cdots+Y_{k_1-1}$. For $t<1-\frac{1+o(1)}{dn}$ we have implies that w.h.p. for $m=k_1-1$,
\begin{align}
\E(e^{tdnZ_1})&=\E\brac{\prod_{i=1}^{m}e^{tdnY_i}}= \sum_{x}\E\brac{\prod_{i=1}^{m}e^{tdnY_i}\mid\g_m=x}\Pr(\g_m=x)\nn\\ 
 &=\E\brac{\prod_{i=1}^{m-1}e^{tdnY_i}}\sum_{x} \E(e^{tdY_m}\mid \g_m=x)\Pr(\g_m=x)\label{middle}\\
&= \E\brac{\prod_{i=1}^{m-1}e^{tdnY_i}}\sum_{x} \frac{x}{x-tdn}\Pr(\g_m=x)= \E\brac{\prod_{i=1}^{m-1}e^{tdnY_i}}\brac{1-\frac{(1+o(1))t}{i}}^{-1}.\nn
\end{align}
Here the term in \eqref{middle} stems from the fact that given $\g_m$, $Y_m$ is independent of $Y_1,Y_2,\ldots,Y_{m-1}$.

Then for any $\b>0$ we have
\mults{
\Pr\brac{Z_1\geq \frac{\b\log n}{dn}}\leq \E(e^{tdnZ_1-t\b\log n}) \leq e^{-t\b\log n}\prod_{i=1}^{k_1-1}\brac{1-\frac{(1+o(1))t}{i}}^{-1}\\
=e^{-t\b\log n}\exp\set{\sum_{i=1}^{k_1-1}\brac{\frac{(1+o(1))t}{i}+O\bfrac{1}{i^2}}} = \exp\set{\brac{1+o(1)-\b}t\log n}. 
}
It follows, on taking $\b=2+o(1)$ that w.h.p. 
\[
d_{j,k_1}\leq \frac{(2+o(1))\log n}{dn}\text{ for all }j\in [n]. 
\]
Letting $\widehat{T}_{k_1}$ be the set corresponding to $T_{k_1}$ when we execute Dijkstra's algorithm starting at vertex 2. First consider the case where $d\leq 1/2$ and \eqref{0} holds. Then, using \eqref{0},  we have that either $T_{k_1}\cap \widehat{T}_{k_1}\neq \emptyset$ or, 
\beq{SP2}{
\Pr\brac{\not\exists e\in T_{k_1}:\widehat{T}_{k_1}:X(e)\leq \frac{1}{n}}\leq \exp\set{-\frac{\psi\th^2n^2}{n}}=o(n^{-2})
}
This shows that we fail to find a path of length $\leq \frac{(4+o(1))\log n}{dn}+\frac{1}{n}$ between a fixed pair of vertices with probability $o(n^2)$.  In particular, taking a union bound over all pairs of vertices, we obtain that w.h.p. $\max_{i,j}d_{i,j}\leq \frac{(4+o(1))\log n}{dn}+\frac{1}{n}$. 

If $G$ has $\delta(G)\geq (1-\t)dn$ with $d=1/2+\e$, $\e>0$ constant, then any pair of vertices has at least $(2\e -2\th)n$ common neighbors. We pair up the vertices of $T_{k_1}$ ${T_{k_2}}$ and bound the probabibility that we cannot find a path of length 2 whose endpoints consist of one of our pairs, and which uses only edges of length at most $\frac{\log n}{n\log\log n}$, as
\[
\left(e^{-(\frac{\log n}{n\log \log n})^2}\right)^{-\th n(2\e n-2\th n)}=o(n^{-2}).
\]
Again we are done by a union bound over possible pairs.
\end{proof}
We now consider the probability that a fixed edge $e$ satisfies that $\ell(e)>\ell_0$ and that $e$ is a shortest path from 1 to $n$.
\begin{lemma}\label{lem5}
Let $\cE(e)$ denote the event that $\ell(e)>\ell_0$ and $e$ is a shortest path from 1 to $n$.
\[
\Pr\brac{\cE\ \bigg| \max_{j}d_{1,j}\leq \ell_1}=o\bfrac{\log n}{n}.
\]
\end{lemma}
\begin{proof}
  Without loss of generality we write $e=\{1,n\}$.  
If $\cE=\cE(e)$ occurs then we have the occurence of the event $\cF$ where 
\[
\cF=\set{d_{1,m}+\ell(f_m)\geq \ell(e),\,m=2,3,\ldots,n-1}
\]
and $f_m$ denotes the edge joining vertex $n$ to the vertex whose shortest distance from vertex 1 (in $G-\set{n}$) is the $m$th smallest. (If the edge does not exist then $\ell(f_m)=\infty$ in the calculation below.) Indeed this follows from Dijkstra's algorithm; the event $\cF$ indicates that at every step of the algorithm, no path shorter than the edge $\{1,n\}$ is found. 

Let $n_0=n(1-d/2)$. We need $\ell(f_m)+d_m\geq\xi=\ell(e)$ for all $m$ in order that $\cF$ occurs. If $d_{1,n_0}=x$ then this is implied by $\bigcap_{m=1}^{n_0}\set{\ell(f_m)\geq \xi-x}$. Using the independence of the $\ell(f_m)$ and $d_{1,i},i=2,\ldots,n_0$, we bound 
\begin{align}
\Pr(\cF\mid \max_{1,j}d_{1,j}\leq \ell_1)&\leq\frac{1}{\Pr(\max_{j}d_{1,j}\leq \ell_1)} \int_{\xi=\ell_0}^{\ell_1}e^{-\xi} \int_{x=0}^\infty\Pr\brac{\bigcap_{m=1}^{n_0}\set{\ell(f_m)\geq \xi-x}}d\Pr\set{d_{1,n_0}=x}d\xi\label{eq4}\\
\noalign{and using the fact that there are at least $dn/2-1$ indices $m$ for which $\ell(f_m)<\infty$ we bound}
\Pr(\cF\mid \max_{1,j}d_{1,j}\leq \ell_1)&\leq  (1+o(1))\int_{\xi=\ell_0}^{\ell_1}\int_{x=0}^\infty \min\set{1,e^{-dn(\xi-x)/3}} d\Pr\set{d_{1,n_0}=x}d\xi. \label{eq4a}
\end{align}
Now, if $\ell_2=\ell_0-\frac{(\log\log n)^2}{dn}$ then
\beq{eq5}{
\int_{\xi=\ell_0}^{\ell_1} \int_{x=0}^{\ell_2} \min\set{1,e^{-dn(\xi-x)/3}} d\Pr\brac{d_{1,n_0}=x}d\xi\leq \ell_1\exp\set{-\frac{(\log\log n)^2}{3}}=o\bfrac{\log n}{n}.
}

It remains to bound the same expression where the second integral goes from $x=\ell_2$ to $\infty$.

First consider the case where $d\leq 1/2$ and \eqref{0} holds. We have from \eqref{Sn2} that 
\begin{align}
\E(d_{1,n_0})&=S_{n_0}\leq (1+O(\th))\sum_{i=1}^{\th n}\frac{1}{dni}+\sum_{i=\th n+1}^{n_0}\frac{1}{\psi i(n-i)}\label{00}\\
&\leq \frac{(1+O(\th))\log n}{dn}+\frac{1}{\psi n}\sum_{i=\th n+1}^{n_0}\brac{\frac{1}{i}+\frac{1}{n-i}} =\frac{(1+O(\th))\log n}{dn}+O\bfrac{\log\log n}{\psi n} \nn\\
&=\frac{\log n}{dn}+O\bfrac{\log^{1/2}n}{n}<\ell_2-\frac{\sqrt{\th}}{2dn}\nn
\end{align}
and 
\beq{000}{
\Var(d_{1,n_0})\leq (1+O(\th))\sum_{i=1}^{\th n}\frac{1}{d^2n^2i^2}+\sum_{i=\th n+1}^{n_0}\frac{1}{\psi^2i^2(n-i)^2}\leq \frac{\p^2}{3d^2n^2}.
}
Chebychev's inequality then gives that
\[
\Pr(d_{1,n_0}\geq S_{n_0}+x)\leq \frac{\p^2}{3d^2x^2n^2}.
\]
As a consequence of this we see that
\beq{eq6}{
\int_{\xi=\ell_0}^{\ell_1} \int_{x=\ell_2}^\infty \min\set{1,e^{-dn(\xi-x)/3}} d\Pr\brac{d_{1,n_0}=x}d\xi\leq \frac{\ell_1\p^2}{3d^2(\ell_2-S_{n_0})^2n^2} \leq \frac{2\ell_1\p^2}{\th\log^2n} =O\bfrac{1}{n\log^{1/2}n}.
}
The lemma follows for $d\leq1/2$, from \eqref{eq5} and \eqref{eq6} and the Markov inequality.

When $d>1/2$ we can replace the second sum in \eqref{00} by 
\[
\sum_{i=\th n+1}^{n_0}\frac{1}{\e n\min\set{i,n-i}}=O\bfrac{1}{n\log n},\qquad\text{ where }\e=d-\frac12.
\]
By the same token, the second sum in \eqref{000} will be $o(n^{-2})$. The remainder of the proof will go as for the case $d\leq 1/2$.
\end{proof}
Together with Lemma \ref{maxdij}, Lemma \ref{lem5} implies that w.h.p. the number of edges $e$ for which $\cE(e)$ occurs is $o(n\log n)$. Adding these to $E_0$ gives us a 1-spanner of size $\approx \frac12n\log n$.
\subsection{Lower bound for part (b)}
\begin{lemma}\label{lem6}
Fix a set $A$ such that $|A|\leq a_0=O(\log n)$. Let $\cP$ be the event that there exists a path $P$  of length at most $\ell_4=\frac{\log n}{200dn}$ joining two distinct vertices of $A$. Then $\Pr(\cP)=O(n^{o(1)-199/200})$.  
\end{lemma}
\begin{proof}
\mults{
\Pr(\cP)\leq a_0^2\sum_{k=0}^{n}((1+\th)dn)^k\frac{\ell_4^{k+1}}{k!}\leq  a_0^2\ell_4 \sum_{k=0}^{n} \bfrac{e^{1+\th}\log n}{200k}^k\leq \\
 a_0^2\ell_4\sum_{k=0}^{\log n}\bfrac{e^{1+\th}\log n}{200k}^k +O(n^{-2})\leq 2a_0^2\ell_4n^{(1+o(1))/200}=O(n^{o(1)-199/200}).
}
\end{proof}
\begin{lemma}\label{lem7}
Let $B_1$ denote the set of vertices whose incident edges of length smaller than $\ell_3=\ell_4/\l$ do not number in the range $I=\left[ \frac{\log n}{300d\l},\frac{\log n}{100d\l}\right]$. Then, w.h.p. $|B_1|\leq n^{1-1/5000\l}$. (Recall that we are bounding the size of a $\l$-spanner from below.)
\end{lemma}
\begin{proof}
The Chernoff bounds imply that 
\mults{
\Pr(v\in B_1)\leq \Pr\brac{Bin\brac{(1\pm\th)dn,1-\exp\set{-\frac{\log n}{200\l dn}}}\notin I }=\\ \Pr\brac{Bin\brac{(1\pm\th)dn,\frac{\log n}{200\l dn}+O\bfrac{\log^2n}{n^2}}\notin I }\leq 2\exp\set{-\frac{(1+o(1))\log n}{2\times 9\times 200\l}} \leq n^{-1/4000\l}.
}
The result follows from the Markov inequality.
\end{proof}
\begin{lemma}\label{lem7a}
Let $B_2$ denote the set of vertices $v$ for which $|\set{w:\ell_{v,w}\leq \ell_4}|\geq \log n$. Then $B_2=\emptyset$ w.h.p.
\end{lemma}
\begin{proof}
The Chernoff bounds imply that 
\[
\Pr(B_2\neq\emptyset)\leq n\Pr\brac{Bin\brac{(1\pm\th)dn,1-\exp\set{-\frac{\log n}{200dn}}}\geq \log n }=o(1).
\]
\end{proof}
Let $B_3$ denote the set of vertices $v$ for which there is a path of length at most $\ell_4$ joining neighbors $w_1,w_2$ such that $\ell_{v,w_i}\leq \ell_3,i=1,2$. Lemma \ref{lem6} with $A$ equal to the set of neighbors $w$ of vertex $v$ such that $\ell_{v,w}\leq \ell_3$ shows that $|B_3|=o(n)$ w.h.p. (The fact that we can take $|A|=O(\log n)$ follows from Lemma \ref{lem2}.) Lemmas \ref{lem7} and \ref{lem7a} then imply that if $v\notin B_1\cup B_3$ then a $\l$-spanner has to include the at least $\log n/(300d\l)$ edges incident to $v$ that are of length at most $\ell_3$. This completes the proof of part (b) of Theorem \ref{th1}.
\section{Summary and open questions}
We have determined the asymptotic size of the smallest 1-spanner when the edges of a dense (asymptotically) regular graph $G$ are given independent lengths distributed as $E_2$, modulo the truth of \eqref{0} or the degree being $dn,d>1/2$. 

There are a number of related questions one can tackle:
\begin{enumerate}
\item We could replace edge lengths by $E_2^s$ where $s<1$. This would allow us to generalise edge lengths to distributions with a density $f$ for which $f(x)\approx x^{1/s}$ as $x\to 0$. This is a more difficult case than $s=1$ and it was considered by Bahmidi and van der Hofstadt \cite{BH}. They prove that w.h.p. $d_{1,2}$ grows like $\frac{n^{s}}{\G(1+1/s)^s}$ where $\G$ denotes Euler's Gamma function.  The analysis is more complex than that of \cite{Jan} and it is not clear that our proof ideas can be generalised to handle this situation.
\item The results of Theorem \ref{th1} apply to $G_{n,p}$. It would be of some interest to consider other models of random or quasi-random graphs.
\end{enumerate}

\end{document}